\documentclass[a4paper]{amsart}
\usepackage{amsfonts}
\usepackage{amssymb}
\usepackage{amsmath}
\usepackage{mathrsfs}
\usepackage{graphicx}
\usepackage{amsmath,amsthm,amssymb,amscd}
\usepackage{xcolor}
\usepackage{enumerate}
\usepackage[T1]{fontenc}
\usepackage{xcolor}

\usepackage{amsmath,amsthm,amssymb,amscd}
\usepackage[all]{xy}
\usepackage[english]{babel}

\newtheorem{thm}{Theorem}[section]
\newtheorem{lem}[thm]{Lemma}

\newtheorem{cor}[thm]{Corollary}
\newtheorem{rem} [thm]{Remark}

\theoremstyle{definition}

\newtheorem{que}[thm]{Question}

\DeclareMathOperator{\diam}{diam}

\DeclareMathOperator{\Int}{Int}

\DeclareMathOperator{\id}{id}

\newcommand{\bS}{\mathbb{S}}

\newcommand{\eps}{\varepsilon}

\newcommand{\A}{\mathbb{A}}

\begin{document}

\title[A Note on Homeo-Product-Minimality]{A Note on Homeo-Product-Minimality}

\author[J. P. Boroński]{J. P. Boroński}\address[J. P. Boroński]{National Supercomputing Centre IT4Innovations, University of Ostrava,
		IRAFM,
		30. dubna 22, 70103 Ostrava,
		Czech Republic}
\email{jan.boronski@osu.cz}
\author[M. Fory\'s-Krawiec]{Magdalena Fory\'s-Krawiec}\address[M. Fory\'s-Krawiec]{AGH University of Science and Technology, Faculty of Applied
	Mathematics, al. Mickiewicza 30, 30-059 Krak\'ow, Poland}
\email{maforys@agh.edu.pl}
\author[P. Oprocha]{Piotr Oprocha}\address[P. Oprocha]{AGH University of Science and Technology, Faculty of Applied
	Mathematics, al. Mickiewicza 30, 30-059 Krak\'ow, Poland  -- and --  National Supercomputing Centre IT4Innovations, University of Ostrava,
		IRAFM,
		30. dubna 22, 70103 Ostrava,
		Czech Republic}
\email{oprocha@agh.edu.pl}

\maketitle	

\begin{abstract}
A compact space $Y$ is called homeo-product-minimal if given any minimal system $(X,T)$, it admits a homeomorphism $S:Y\to Y$, such that the product system $(X\times Y,T\times S)$ is minimal. We show that a large class of cofrontiers is homeo-product-minimal. This class contains R. H. Bing's pseudo-circle, answering a question of Dirb\'{a}k, Snoha and \v Spitalsk\'{y} from [{\it Minimal direct products}, Trans. Amer. Math. Soc.
375 (2022)].
\end{abstract}

\section{Introduction}
The present paper concerns the notion of homeo-product-minimality, recently introduced by Dirb\'{a}k, Snoha and \v Spitalsk\'{y} in \cite{DSS}. A compact space $Y$ is called \emph{homeo-product-minimal} if given any minimal, metric but not necessarily compact system $(X,T)$, $Y$ admits a homeomorphism $S:Y\to Y$, such that the product system $(X\times Y,T\times S)$ is minimal. The aforementioned authors showed (\cite[Theorem B]{DSS}) that the following spaces are homeo-product-minimal: 
\begin{enumerate}
    \item every compact connected metrizable abelian group,
    \item every space of the form $Y\times Z$, where $Y$ is a nondegenerate homeo-product-minimal space and $Z$ is a compact metrizable space admitting a minimal action of an arc-wise connected topological group,
    \item every quotient space $(\Gamma\times Z)/\Lambda$, obtained from $\Gamma\times Z$ by applying the diagonal action of $\Lambda$, where $\Gamma$ is an infinite compact connected metrizable abelian group, $\Lambda$ is a finite subgroup of $\Gamma$ and $Z$ is a compact connected (not necessary abelian) metrizable group, on which the group $\Lambda$ acts by automorphism,
    \item the Klein bottle,
    \item the Cantor space,
    \item the Sierpi\'nski curve on the $2$-torus and the Sierpi\'nski curve on the Klein bottle,
    \item every compact metrizable space $Y$ admitting a minimal continuous flow whose centralizer in the group of homeomorphisms of $Y$ acts transitively on $Y$  in the algebraic sense,
    \item every compact connected manifold $Y$ without boundary admitting a free action of a nontrivial compact connected Lie group,
    \item all odd-dimensional spheres and compact connected Lie groups,
    \item every cantoroid.
\end{enumerate}
The introduction of the notion of homeo-product-minimality was motivated by the fact, that there exist minimal spaces, whose Cartesian powers do not admit minimal homeomorphisms \cite{BCO18}. 

Among the open questions raised by the authors in \cite{DSS}, was the one asking if R. H. Bing's pseudo-circle is homeo-product-minimal \cite[Problem 4]{DSS}. The pseudo-circle is a fractal-like object, arising both in topology and dynamics. For example, $\mathbb{R}^2\setminus\{(0,0)\}$ admits a continuous decomposition into pseudo-circles, whose decomposition space is homeomorphic with $\mathbb{R}$ \cite{Brown} (see also \cite{BCJ})). In dynamics, it appears in rotation theory \cite{BO}, as well as smooth \cite{KennedyYorke1}, \cite{KennedyYorke2}, \cite{KennedyYorke3}, \cite{Hn}, and complex dynamics \cite{Cheritat}. The question concerning homeo-product-minimality of the pseudo-circle seems well motivated, since the pseudo-circle admits minimal homeomorphisms \cite{Hn} and minimal noninvertible maps \cite{BKLO}. It also shares many properties with the circle, admitting homeomorphisms that are well-approximated by rigid rotations of the annulus. However, its homeomorphism group is far from the one of the circle. For example, the action of the group on the pseudo-circle has uncountably many orbits \cite{KR}, contains minimal homeomorphisms that are weakly mixing \cite{Hn} (unlike the circle), and does not contain any connected nondegenerate compact subsets \cite{Lewis}. Nonetheless, in the present paper we show that the pseudo-circle is indeed homeo-product-minimal, a property it shares with a class of planar cofrontiers, the so-called HAK attractors \cite{BCO}. In the next section we recall the construction of these objects, which we shall use in Section 3 to answer the aforementioned question. As noted by Handel in \cite{Hn}, the construction is flexible enough, to allow the introduction of crookedness, which guarantees that the resulting HAK attractor is a pseudo-circle, giving us the following result, which in turn provides an affirmative answer to the \cite[Problem 4]{DSS} mentioned above.
\begin{thm}\label{main}
The pseudo-circle is homeo-product-minimal. 
\end{thm}

\section{Preliminaries}\label{sec2}
 A \emph{dynamical system} is a pair $(X,T)$, where $X$ is metrizable and $T:X\to X$ is a continuous map. In particular we do not require the usual assumption that $X$ is a compact space. We adapt this more general approach from \cite{DSS} since definition of homeo-product-minimality require it. Dynamical system is \emph{minimal} if there is no proper nonempty, closed and $T$-invariant subset $M\subset X$, in which case map $T$ is called minimal. Note that a system is minimal if and only if \emph{$\omega$-limit set of a point $x \in X$}, denoted $\omega_T(x)$, is the whole space $X$ for all $x \in X$, where $\omega_T(x)$ is the set of all limit points of the trajectory $\{T^n(x)\}_{n \geq 0}$. Recall that a compact space $Y$ is called \emph{homeo-product-minimal} if given any minimal, metric but not necessarily compact system $(X,T)$, $Y$ admits a homeomorphism $S:Y\to Y$, such that the product system $(X\times Y,T\times S)$ is minimal. 

By a \emph{continuum} we mean compact, connected metric space and any subset of a continuum which is a continuum itself is referred to as \emph{subcontinuum}. A continuum is \emph{indecomposable} if it cannot be presented as the union of two proper subcontinua and \emph{hereditarily indecomposable} if every compact connected subset of a continuum is indecomposable. The main object of our interest is the pseudo-circle constructed by Bing in \cite{B2}, which is a~hereditarily indecomposable continuum.

Let $\bS^1$ denote the unit circle. For any $\theta \in \mathbb{R}$ and the circle $\bS^1$ we denote by $R_{\theta}:\bS^1 \ni x\mapsto x+\theta\mod 1 \in  \bS^1$ the rotation map. It is clear that if we fix some $\alpha \in \mathbb{R}\setminus\mathbb{Q}$ and a sequence $\{q_k\}_{k\geq 0}\subset\mathbb{Q}$ converging to $\alpha$, then the sequence of rotations $\{R_{q_k}\}_{k\geq 0}$ converges to $R_{\alpha}$ in the space of self homeomorphisms of the circle.

Take the annulus $\A = [0,1]\times \mathbb{R}/\mathbb{Z}$ and by $\pi_i$, $i=1,2$ denote the projection of $\A$ onto its $i$-th factor; that is for any $t = (t_1,t_2) \in \A$ we have $\pi_1(t) = t_1 \in [0,1]$ and $\pi_2(t)=t_2\in \mathbb{R}/\mathbb{Z}$. 

For any compact metric space $(X,d)$ and continuous maps $F,G:X\to X$ we define the following metric:
	$$
	\rho(F,G) = \sup_{x \in X}d(F(x),G(x)).
	$$
	 
The proof of our main Theorem \ref{thm:pseudocirc} is motivated by the construction of so-called HAK homeomorphism (Handel-Anosov-Katok homeomorphism) of annulus $\A$. We say that a homeomorphism $H\colon\A\to \A$ is a \emph{HAK homeomorphism} if $H$ is the uniform limit (in terms of the metric $\rho$) of a sequence of particular annulus homeomorphisms $\{H_n\}_{n\geq 0}$. Below we recall the main ideas of the construction of sequence $\{H_n\}_{n \geq 0}$ after the authors of \cite{BCO}. For more details we refer the reader to \cite{Hn} and \cite{BCO}.

{\bf Construction of HAK homeomorphism:}
\begin{enumerate}
	\item Start with a decreasing sequence of closed annuli $\{A_n\}_{n \in \mathbb{N}}$ assuming $A_0=\A$, such that $A_1\subset \Int \A$. For every $n \in \mathbb{N}$ let $f_n:A_n\to \A$ be a homeomorphism with the property that $\pi_1(f(c))\leq \frac{\eps_n}{2}$ for $c \in A_n$ and $\eps_n$ small enough. In fact we choose $\eps_n$ for $n \in \mathbb{N}$ in such a way that $\sum_{i=0}^{\infty}\eps_i$ is convergent (see section 2.1.1 in \cite{BCO}).  
	\item For each $n$ we take the rational rotation $R_n:\A\to\A$ of period $p_n$ such that there exists a rectangle $D_n = [0,1]\times [0,\alpha_n]$ with the property that $R^i_n(D_n), R^j_n(D_n)$ are either equal or have disjoint interiors, $\A = \bigcup_{i \in \mathbb{Z}}R^i_n(D_n)$ and $\alpha_n<\delta_n$, where $\delta_n$ is chosen by the uniform continuiuty of $f_n^{-1}$ for $\frac{\eps_n}{4}$. We assume that $p_{n+1}<p_n$ for any $n \in \mathbb{N}$. 
	\item For each $n \in \mathbb{N}$ we define a homeomorphism $g_n:\A\to\A$ which is $\eps_n$-close to identity and:
	$$
	g_n|_{A_n} = f_n^{-1}\circ R_n \circ f_n,
	$$
	while $g_n|_{\A\setminus A_{n-1}}=\id$. 
	\item For every $n \in \mathbb{N}$ we define the homeomorphism $H_n:\A\to\A $ as follows:
	$$
	H_n = g_n\circ g_{n-1}\circ \dots\circ g_1.
	$$
	\end{enumerate}
The approximation scheme presented in \cite{BCO} assures also that $H_n(A_{n+1})\subset \Int A_{n+1}$ and $\rho(H^i_n, H^i_{n+1})<\eps_n$ for $i=1\dots,mp_n$ for some $m\geq 1$. As a result, the authors obtain that $\{H_n\}_{n \in \mathbb{N}}$ is a Cauchy sequence,  so it converges to some homeomorphism $H$, together with $H$-invariance of each $A_n.$ Denoting $\mathcal{C} = \bigcap_{n \in \mathbb{N}}A_n$ we have that $\mathcal{C}$ is closed and $H$-invariant as well. In fact it is an attracting set of $H$ and we will refer to it as \emph{HAK attractor} of the \emph{HAK homeomorphism} $H$. 

\section{Homeo-product-minimality}

Throughout this section we always assume that $(X,T)$ is a minimal dynamical system, where $X$ is metric, but not necessarily compact.

The following Lemma \ref{lem: FG} is standard. We state it without proof.
\begin{lem}\label{lem: FG}
	Let $F\colon Z\to Z$ be a continuous map on a compact metric space $Z$. Then for any continuous map $G\colon Z\to Z$, $K>0$ and $\eps>0$ there exists $\gamma>0$ such that if $\rho(F,G)<\gamma$ then $\rho(F^j,G^j)<\eps$ for $j=0,\dots K$.
\end{lem}

\begin{rem}\label{rem:alpha}
By the Weyl's theorem (see for example \cite{KN}) we have that for any sequence $\{x_n\}_{n \in \mathbb{N}}$ of distinct integers the sequence $\{\alpha x_n\}_{n \in \mathbb{N}}$ is uniformly distributed $\mod 1$ for almost all $\alpha \in \mathbb{R}$. Using this result and repeating the reasoning presented in the proof of \cite[Proposition 1]{KST} we have that the set of $\alpha \in \mathbb{R}\setminus\mathbb{Q}$ such that $T\times R_{\alpha}$ is minimal is of full Lebesgue measure, where $(X,T)$ is a minimal dynamical system and $R_{\alpha}:\bS^1\to \bS^1$ is the irrational rotation of the circle. 
\end{rem}

The following fact is a standard consequence of minimality. We present its proof for completeness.

\begin{lem}\label{lem:circle}

Fix any $x_0\in X$. Let $R_{\alpha}:\bS^1\to \bS^1$ be a rotation on the circle $\bS^1$, with $\alpha \in \mathbb{R}\setminus \mathbb{Q}$ such that $T\times R_{\alpha}$ is minimal. For any $\eps>0$ there exists $K >0$ such that for every $t, s \in \bS^1$ there is $j<K$ such that:
	$$
	d\left((x_0,t), (T\times R_{\alpha})^j(x_0,s))\right)<\frac{\eps}{2}.
	$$
\end{lem}
\begin{proof}
	Fix $\eps>0$ and take any $t,s \in \bS^1$. Let $\{V_i\}_{i=1}^N$ be a finite open cover of $\bS^1$ with sets of diameter less than $\frac{\eps}{2}$
 and let $W$ be a neighborhood of $x_0$ with $\diam W<\frac{\eps}{2}$ (for simplicity we assume that in $X\times \bS^1$ we have metric maximum). Denote $U_i=W\times V_i$. By minimality of $T\times R_{\alpha}$ for each $i=1,\dots,N$  there exist $K_i$ and  some $j_i<K_i$ such that $(T\times R_{\alpha})^{j_i}(x_0,s) \in U_i$. Let $K = \max\{K_i: i=1,\dots, N\}$ and by $U \in \{U_i\}_{i=1}^N$ denote the set containing $(x_0,t)$. Then for some $j \in \{j_1,\dots, j_N\}$ we have $d\left((x_0,t), (T\times R_{\alpha})^j(x_0,s)\right)<\frac{\eps}{2}$.
\end{proof}

\begin{lem}\label{lem:TxRpq}
Fix any $x_0\in X$. For all $\eps>0$, $\frac{p}{q} \in \mathbb{Q}$, any $m \in \mathbb{N}$ and $\delta>0$ such that $\delta<\frac{1}{m}$ there exists a rational number $\frac{p'}{q'} \in \mathbb{Q}\cap [1-\delta,1+\delta]$ with $q'>m$ and an integer $K >0$ with the property that:
$$	d\left((x_0,t), (T\times R_{\frac{p}{q}\cdot \frac{p'}{q'}})^j(x_0,s)\right)< \frac\eps2.
	$$
for every $t, s \in X\times \bS^1$ and some $j=j(s,t)<K$.
\end{lem}
\begin{proof}
	Fix $\eps>0$, $\frac{p}{q} \in \mathbb{Q}$, some arbitrarily large $m \in \mathbb{N}$,  and $0<\delta<\frac{1}{m}$. 
	Choose some irrational $\alpha \in (1-\delta,1+\delta)$ (such an $\alpha$ exists by Remark~\ref{rem:alpha}) such that $T\times R_{\frac{p}{q}\cdot\alpha}:X\times \bS^1\to X\times \bS^1$ is minimal.
	Let $K$ be given by Lemma \ref{lem:circle} for $\frac{\eps}{3}$ and the map $T\times R_{\frac{p}{q}\cdot\alpha}$. By Lemma \ref{lem: FG} for $T\times R_{\frac{p}{q}}$, the chosen $K$ and $\frac{\eps}{3}$ there exists $\gamma>0$ such that for $\beta \in (1-\delta,1+\delta)$ we have $\rho((T\times R_{\frac{p}{q}\cdot \alpha})^j, (T\times R_{\frac{p}{q}\cdot \beta})^j)<\frac{\eps}{3}$ for $j=0,\dots,K$ whenever $\rho((T\times R_{\frac{p}{q}\cdot \alpha}), (T\times R_{\frac{p}{q}\cdot \beta}))<\gamma$. Now choose $\delta'<\delta$ such that for any $\beta \in (\alpha-\delta',\alpha+\delta')$ we have $\rho((T\times R_{\frac{p}{q}\cdot \alpha}), (T\times R_{\frac{p}{q}\cdot \beta}))<\gamma$ and $(\alpha-\delta',\alpha+\delta') \subset (1-\delta, 1+\delta)$.

Take any rational number $\frac{p'}{q'} \in \mathbb{Q}$ with the property that:
	$$
	\left| \frac{p'}{q'}-\alpha\right|<\delta'.
	$$
Note that $\frac{p'}{q'} \in (1-\delta,1+\delta)$ and clearly
	$$
	\frac{1}{q'}\leq \left| \frac{p'}{q'}-1\right|<\delta<\frac{1}{m},
	$$
hence $q'>m$.
	Furthermore, we have:
	\begin{multline*}
	d\left((x_0,t),(T\times R_{\frac{p}{q}\cdot\frac{p'}{q'}})^j(x_0,s)\right)\leq d\left((x_0,t),(T\times R_{\frac{p}{q}\cdot\alpha})^j(x_0,s), )\right)\\ 
	+ d\left((T\times R_{\frac{p}{q}\cdot\alpha})^j(x_0,s), (T\times R_{\frac{p}{q}\cdot\frac{p'}{q'}})^j(x_0,s)\right)
	 < \frac{\eps}{6}+\frac{\eps}{3}= \frac\eps2.
\end{multline*}
The proof is complete.
 \end{proof}

In the proof of the theorem presented below we will use the approximation scheme of HAK homeomorphism $H:\A\to\A$ based on the one presented in section \ref{sec2}.

\begin{thm}\label{thm:pseudocirc}
Fix any $x_0\in X$. There is a nested sequence of annuli $\{A_n\}_{n\geq 0}\subset\A$ and a homeomorphism $H:\mathbb{A}\to\mathbb{A}$ such that $\mathcal{C} = \bigcap_{n\geq 0}A_n\subset \Int \mathbb{A}$ is an attracting pseudo-circle and for any $n> 0$ there exists $K_n >0$ such that for every $t,s \in \A$ there is $j_n<K_n$ such that:
$$
d\left((x_0,t), (T\times H)^{j_n}(x_0,s)\right)<\frac1n.
$$
\end{thm}
\begin{proof}
	Choose an arbitrarily large $m_1 \in \mathbb{N}$ such that $\xi_1 = \frac{1}{m_1}<\frac14$, take any $p \in \mathbb{N}$ and define the annulus $\hat{\A} = [\frac14,\frac34]\times \mathbb{R}/\mathbb{Z}\subset \A$. For any $i=1,\dots,m_1$ let $P^1_i = \left[\frac14,\frac34\right]\times \left[\frac{i-1}{m_1}, \frac{i}{m_1}\right]$ and $\mathcal{P}_1 = \{P^1_1,\dots, P^1_{m_1}\}$. Take an annulus $A_1\subset \A$ 
 and a~homeomorphism $f_1:\A\to \A$ with $f(A_1) = \hat{\A}$ such that if we denote 
 $D^1_i = f^{-1}_1(P^1_i)$, then 
 $A_1$ is thin enough so that $\diam (D^1_i)<\frac14$ for $i=1,\dots,m_1$.

	Fix $\eps_1<\frac14$. Let $\frac{p'_1}{q'_1}$, $K_1$ and $j_1<K_1$ be obtained by Lemma \ref{lem:TxRpq} for $m_1, \eps_1$, $\frac{p}{m_1}$ and some $\delta_1<\frac{1}{m_1}$. Take $\frac{p_1}{q_1} = \frac{pp'_1}{m_1q'_1}$. 
	
	 Define an annulus homeomorphism $g_1:\A\to\A$ as follows: 
	 $$g_1 = f_1^{-1}\circ \left(\id_{[0,1]}\times R_{\frac{p_1}{q_1}}\right)\circ f_1.$$
	Let $H_1 = g_1$. 
Observe that for any $j\geq 0$ and $s \in A_1$ we have:
$$
f_1(H_1^j(s)) = \left(\id_{[0,1]}\times R_{\frac{p_1}{q_1}}\right)^j(f_1(s)).
$$
Fix $t,s\in A_1$.
Denote $\hat{t} = f_1(t)$ and $\hat{s}= f_1(s)$. There exists an $i \in \{1,\dots, m_1\}$ such that $\hat{s}, \left(\frac14,\pi_2(\hat{s})\right)\in P^1_i$ and  for $j_1<K_1$ we have:
	 $$
	 d\left((x_0,\pi_2(\hat{t})),\left(T\times R_{\frac{p_1}{q_1}}\right)^{j_1}( x_0,\pi_2(\hat{s}))\right)<\frac{\eps_1}{2}.
	 $$
	 It follows that $\left(\frac14,\pi_2(\hat{t})\right)$ and $\left(\frac14, R_{\frac{p_1}{q_1}}^{j_1}(\pi_2(\hat{s}))\right)$ are in the same or neighbouring regions of the partition $\mathcal{P}_1$, which means that $\hat{t}$ and $f_1((H_1^{j_1})(s))$ are in the same or neighbouring regions of partition $\mathcal{P}_1$ as well. It follows that $d(t, H_1^{j_1}(s))<\frac12$. Hence:
	 $$
	 d\left((x_0,t),(T\times H_1)^{j_1}( x_0,s)\right)<\frac12+\eps_1<\frac34. 
	 $$
	 
	 Let $\gamma_1$ be given by Lemma \ref{lem: FG} for $H_1$ and $\frac{\eps_1}{2}$. 
	 Take $\eta_1 \leq \frac{\gamma_1}{4}$.

For any $n> 1$ we proceed with the construction inductively as follows. Suppose we have already defined $A_{n-1}$, $f_{n-1}$, $g_{n-1}$, $H_{n-1}$, $\frac{p_{n-1}}{q_{n-1}}$, $m_{n-1}$, $\eps_{n-1}$. Let $\xi_n>0$ be such that if $\diam(Z)<\xi_n$ then $\diam \left(f_{n-1}^{-1}(Z)\right)<\frac{1}{4n}$ for any $Z\subset A_{n-1}$. Consider $\frac{p_{n-1}}{q_{n-1}}$ from the previous step.  Fix an annulus $A_n\subset \Int A_{n-1}$ by taking some curve in $A_{n-1}$ which is invariant under $\left(\id_{[0,1]} \times R_{\frac{p_{n-1}}{q_{n-1}}}\right)$ and thickening it to an annulus $A_{n}$, which is  invariant under $\left(\id_{[0,1]} \times R_{\frac{p_{n-1}}{q_{n-1}}}\right)$ as well. Define a~homeomorphism $\varphi_n: \A\to \A$ starting with 
a homeomorphism $\varphi_n:A_n\to\hat{\A}$ satisfying
$$
\varphi_n\circ \left. \left(\id_{[0,1]} \times R_{\frac{p_{n-1}}{q_{n-1}}}\right)\right|_{A_n} = \left(\id_{[0,1]} \times R_{\frac{p_{n-1}}{q_{n-1}}}\right)\circ \left.\varphi_n\right|_{A_n}
$$
and extending it to the whole $\A$. Let $f_n:\A\to \A$ be as follows:
$$
f_n = \varphi_n\circ f_{n-1}.
$$
This way $f_n$ is commuting with $\id_{[0,1]} \times R_{\frac{p_{n-1}}{q_{n-1}}}$ on $A_n$.
We also require that the annulus $A_n$ is thin enough so that there exists some $\beta_n>0$ such that if we take a~region $\left[\frac14,\frac34\right]\times[x,x+\beta_n]\subset \hat{\A}$ then: $$\diam\left(\varphi_n^{-1}\left(\left[\frac14,\frac34\right]\times[x,x+\beta_n]\right)\right) <\xi_n$$
 for any $x \in \mathbb{R}/\mathbb{Z}$. 

Fix some $\eps_n<\frac{1}{4n}$. Take some sufficiently large $\alpha_n \in \mathbb{N}$ so that for $m_n = m_{n-1}\alpha_n > m_{n-1}q'_{n-1}$ we have $\frac{1}{m_n}<\beta_n$. Now take $K_n, j_n<K_n$, $\frac{p'_n}{q'_n}$ obtained by Lemma \ref{lem:TxRpq} for $\eps_n$, $m_n$, some $\delta_n<\frac{1}{m_n}$ and $\frac{p_{n-1}}{q_{n-1}}$. Take $\frac{p_n}{q_n} = \frac{p_{n-1}}{q_{n-1}}\cdot\frac{p'_n}{q'_n}$ putting $p_n = p_{n-1}p'_n$ and $q_n = q_{n-1}q'_n$ and denote $\nu_n = \frac{p'_n-q'_n}{q'_n}\cdot \frac{p_{n-1}}{q_{n-1}}$. Let $P^n_i = \left[\frac14,\frac34\right]\times \left[\frac{i-1}{q_n}, \frac{i}{q_n}\right]$, then $\mathcal{P}_n = \{P^n_1,\dots, P^n_{q_n}\}$  is a partition of $\A$ into $q_n$ rectangles with the shorter sides of length $\frac{1}{q_n}<\frac{1}{m_n}<\beta_n$. 
Let $D^n_i = f^{-1}_n(P^n_i)$. 

In particular for any $s,t\in A_n$
there is $j<K_n$ for which by Lemma \ref{lem:TxRpq} we have:
$$
d\left(( x_0,\pi_2(t)),(T\times R_{\frac{p_n}{q_n}})^{j}( x_0,\pi_2(s))\right)<\frac{\eps_n}{2}.
$$
Let $g_n:\A\to\A$ be the homeomorphism defined as follows. Start with: 
$$g_n|_{A_n} = f_n^{-1}\circ \left(\id_{[0,1]}\times R_{\nu_n}\right)\circ f_n$$
 and extend it to a homeomorphism on the whole $\A$ which is the identity on the boundary. Let $H_n = g_n\circ g_{n-1}\circ\dots\circ g_1$. In particular $H_n|_{A_n} = g_n\circ H_{n-1}|_{A_n}$.

For any $1\leq j \leq K_n$ and $s \in A_n$  
we have the following, based on the definition of $g_n$ and the inductive assumption for $n-1$:
\begin{align*}
	f_n(H_n^j(s)) &= f_n\circ (g_n\circ H_{n-1})^j(s) = \left(\id_{[0,1]}\times R_{\nu_n}\right)\circ f_n\circ H_{n-1}\circ( g_n\circ H_{n-1})^{j-1}(s) \\
	&= (\id_{[0,1]}\times R_{\nu_n})\circ f_n\circ f_{n-1}^{-1}\circ (\id_{[0,1]}\times R_{\frac{p_{n-1}}{q_{n-1}}})\circ f_{n-1}\circ(g_n\circ H_{n-1})^{j-1}(s) \\
	&= (\id_{[0,1]}\times R_{\nu_n})\circ\varphi_n \circ (\id_{[0,1]}\times R_{\frac{p_{n-1}}{q_{n-1}}})\circ f_{n-1}\circ(g_n\circ H_{n-1})^{j-1}(s) \\
	&= (\id_{[0,1]}\times R_{\nu_n})\circ (\id_{[0,1]}\times R_{\frac{p_{n-1}}{q_{n-1}}})\circ \varphi_n\circ f_{n-1}\circ(g_n\circ H_{n-1})^{j-1}(s) \\
	&= (\id_{[0,1]}\times R_{\frac{p_n}{q_n}})\circ \varphi_n \circ f_{n-1}\circ(g_n\circ H_{n-1})^{j-1}(s) \\
	&= (\id_{[0,1]}\times R_{\frac{p_n}{q_n}})\circ f_n\circ(g_n\circ H_{n-1})^{j-1}(s)\\
	&= (\id_{[0,1]}\times R_{\frac{p_n}{q_n}})\circ f_n\circ f_{n-1}^{-1}\circ (\id_{[0,1]}\times R_{\frac{p_{n-1}}{q_{n-1}}})\circ f_{n-1} \circ (g_n\circ H_{n-1})^{j-2}(s)\\
	& = \dots = (\id_{[0,1]}\times R_{\frac{p_n}{q_n}})^j\circ f_n(s) = (\id_{[0,1]}\times R_{\frac{p_n}{q_n}})^j(f_n(s)).
\end{align*}

Fix $s,t\in A_n$. Denote $\hat{t} = f_n(t)$ and $\hat{s} = f_n(s)$. Then $\hat{s}$ and $(\frac14,\pi_2(\hat{s}))$ are both in the same of $q_n$ regions of $f_n^{-1}(\hat{\A})$. 
For $j_n<K_n$ we have:
$$
d\left(( x_0,\pi_2(\hat{t})), (T\times R_{\frac{p_n}{q_n}})^{j_n}( x_0,\pi_2(\hat{s}))\right)<\frac{\eps_n}{2},
$$
hence $(\frac 14,\pi_2(\hat{t}))$ and $(T\times R_{\frac{p_n}{q_n}})^{j_n}(y,\pi_2(\hat{s}))$ are in the same or neighbouring regions of partition $\mathcal{P}_n$, which implies that $\hat{t}$ and $f_n(H_n^{j_n}(s))$ are in neighbouring regions of the partition $\mathcal{P}_n$ as well. Hence we have $d(t,H_n^{j_n}(s))<\frac{1}{2n}$ and it follows that:
$$
d(( x_0,t), (T\times H_n)^{j_n}( x_0,s))<\frac{1}{2n}+\frac{\eps_n}{2}<\frac{1}{2n}+\eps_n<\frac{3}{4n}.
$$

Let $\gamma_n$ be given by Lemma \ref{lem: FG} such that whenever $\rho(H_n,  \id_{[0,1]}\times R_{\frac{p_n}{q_n}})\leq \gamma_n$ then $\rho(H_n^i,  \id_{[0,1]}\times R^i_{\frac{p_n}{q_n} } )\leq \frac{\eps_n}{2}$ for $i=0,\dots,K_0$. Take $\eta_n \leq \frac{\min\{\gamma_1,\dots, \gamma_{n-1}\}}{2^{n+1}}$.

Note that $\{H_n\}_{n\geq 1}$ is a Cauchy sequence, as for any $n \geq 1$ and $s>n$ we have:
$$
\rho(H_n^i, H_{n+s}^i)\leq \sum_{j=n}^{n+s-1}\rho(H_j^i,H_{j+1}^i)\leq 2^{-n}. 
$$
Similarly one shows that the sequence $\{H^{-1}_n\}_{n \geq 1}$ is a Cauchy sequence as well. Therefore, 
the map $H = \lim_{n\to\infty}H_n$ is a well defined homeomorphism of $\A$. Note that the pseudo-circle $\mathcal{C} = \bigcap_{n\geq 1}A_n\subset \A$ is $H$-invariant. By the results of Handel in \cite{Hn} we have that $H$ is minimal and uniformly rigid.
 As $\rho(H_n, H_{n+1})<\frac{\gamma_n}{2^{n+1}}$ we obtain the following estimate for every $n\geq 1$:
$$
\rho(H_n,H)<\sum_{j=n}^{\infty}\frac{\gamma_n}{2^{j+1}}	\leq \frac{\gamma_n}{2^n}<\gamma_n, 
$$
and by Lemma \ref{lem: FG} for $i=0,\dots, K_n$ we have:
$$
\rho(H_n^i, H^i)<\frac{\eps_n}{2}.
$$
Finally take any $t,s\in \A$ and note that:
\begin{align*}
d(( x_0,t), (T\times H)^{j_n}( x_0,s))& \leq d(( x_0,t), (T\times H_n)^{j_n}( x_0,s))\\
&\qquad + \rho((T\times H_n)^{j_n}( x_0,s),(T\times H)^{j_n}( x_0,s))\\
&\qquad\qquad \leq \frac{3}{4n} + \frac{\eps}{2} <\frac1n.
\end{align*}
 The proof is finished.
\end{proof}

\begin{cor}
	For every minimal system $(X,T)$ there is an embedding of pseudo-circle $\mathcal{C}$ in annulus $\mathbb{A}$ and the homeomorphism $H: \mathbb{A}\to\mathbb{A}$ such that $H(\mathcal{C}) = \mathcal{C}$, $(\mathcal{C},H|_{\mathcal{C}})$ is minimal and $(X\times \mathcal{C}, T\times H|_{\mathcal{C}})$ is minimal.  
\end{cor}
Now, Theorem \ref{main} is an immediate corollary as shown below.
	\begin{proof}[Proof of Theorem \ref{main}]
		The proof is an consequence of Theorem \ref{thm:pseudocirc} and the fact that the pseudo-circle is topologically unique (see \cite{LF}). 
  Namely, it is enough to repeat a standard argument for disjointness (see \cite[Proposition~1]{KST}). Let us fix any $x_0\in X$. Then Theorem \ref{thm:pseudocirc} implies that there is a map $H$ on the pseudocircle $\mathcal{C}$ such that $\omega_{T\times H}(x_0,y)\supset \{x_0\}\times \mathcal{C}$ for every $y\in \mathcal{C}$. But $\omega_{T\times H}(x_0,y)$ is closed, invariant, and $x_0$ has dense orbit in $X$ under $T$, hence $\omega_{T\times H}(x_0,y)=X\times \mathcal{C}$ for any $y\in \mathcal{C}$.
  To complete the proof, observe that by minimality of $T$, for any $(x,t)\in X\times \mathcal{C}$ there is $s\in \mathcal{C}$ such that $(x_0,s)\in \omega_{T\times H}(x,t)$, which implies
  that $X\times \mathcal{C}=\omega_{T\times H}(x_0,s)\subset \omega_{T\times H}(x,t)$.
	\end{proof}

Note that it is still an open question whether all real numbers can be realized as  rotation numbers of minimal pseudo-circle homeomorphisms and, with the exception for integer multiples of the same number, if two such numbers can be realized on the same embedding (see \cite{BCO} for more details). This motivates the following question.
\begin{que}
	Is there an embedding $\mathcal{C}$ of the pseudo-circle into annulus $\A$ such that for any minimal system $(X,T)$ there is a homeomorphism $H:\A\to \A$ such that $H(\mathcal{C}) = \mathcal{C}$, $(\mathcal{C},H)$ is minimal and $(X\times \mathcal{C}, T\times H|_{\mathcal{C}})$ is minimal?  
\end{que} 
Note also that if we denote by $C_n$ the boundary circle of each annulus $A_n$ from the construction in the proof of Theorem \ref{thm:pseudocirc}, then $\{C_n\}_{n\in \mathbb{N}}$ converges in Hausdorff metric to the pseudo-circle $\mathcal{C}$. In particular each $C_n$ is homeo-product-minimal and so is the limit $\mathcal{C}$. On the other hand, in plain it is easy to obtain arc as the limit of converging sequence of circles. This leads to the following question:
\begin{que}
Take any $m \in \mathbb{N}$ and a sequence   $\{Y_n\}_{n\in \mathbb{N}}$ of homeo-product-minimal spaces, where $Y_n \subset \mathbb{R}^m$ for $n \in \mathbb{N}$. Let $Y$ be the limit of $Y_n$ in the Hausdorff metric. Under what conditions $Y$ is a homeo-product-minimal space?
\end{que}

\section*{Acknowledgements}
J. P. Boroński and M. Fory\'s-Krawiec were supported in part by the National Science Centre, Poland (NCN), grant no. 2019/34/E/ST1/00237. P. Oprocha was supported by National Science Centre, Poland (NCN), grant no. 2019/35/B/ST1/02239.

\bibliographystyle{amsplain}

\end{document}